\documentclass[12pt]{amsart}

\usepackage{float}

\usepackage{amssymb}
\usepackage{amscd}
\usepackage{graphicx}
\usepackage{amsmath}
\usepackage{setspace}
\usepackage[all]{xy}
\usepackage{color}
\usepackage{enumerate}

\newtheorem{theorem}{Theorem}[section]

\newtheorem{corollary}[theorem]{Corollary}

\newtheorem{lemma}[theorem]{Lemma}

\newtheorem{proposition}[theorem]{Proposition}

\theoremstyle{remark}

\newtheorem{remark}[theorem]{Remark}

\newcommand{\CS}{\mathcal{S}}

\newcommand{\CW}{\mathcal{W}}

\newcommand{\rmW}{\mathrm{W}}

\newcommand{\BBP}{{\mathbb P}}
\newcommand{\BBQ}{{\mathbb Q}}

\newcommand{\BBZ}{{\mathbb Z}}

\newcommand{\Hom}{\mathrm{Hom}}
\newcommand{\Ima}{\mathrm{Im}}

\newcommand{\Ker}{\mathrm{Ker}}
\newcommand{\Coker}{\mathrm{Coker}}

\newcommand{\End}{\mathrm{End}}

\begin{document}

\title{A mixed version for a Fuchs' Lemma}

\author{Simion Breaz}

\address[bodo@math.ubbcluj.ro]{"Babe\c s-Bolyai" University, Faculty of Mathematics and Computer Science, Str. Mihail Kog\u alniceanu 1, 400084, Cluj-Napoca, Romania}




\begin{abstract}
We prove a version for mixed groups for a Fuchs' result about connections between the cancellation property of a group and the unit lifting property of its (Walk-)endomorphism rings.   
\end{abstract}




\maketitle

\centerline{\textit{Dedicated to L\'aszl\'o Fuchs on the occasion of his 95th anniversary}}

\section{Introduction}
In this paper all groups are abelian. A group $G$ has the \textit{cancellation property} if whenever $H$ and $K$ are  groups such that $G\oplus H\cong G\oplus K$ it follows that $H\cong K$. We refer to \cite{Fuchs-2015} for details about the history of the study of the cancellation property, and to \cite{MS} and \cite{PH} for recent results connected to this property.
An important question is whether we can find classes of groups that have the cancellation property and can they be characterized by using ring theoretical properties of the corresponding endomorphism rings. For instance, L\'aszl\'o Fuchs introduced in \cite{Fuchs71} \textit{the substitution property}: the group $G$ has the substitution property if for every group $A$ which has direct decompositions $A=G_1\oplus H=G_2\oplus K$ such that $G_1\cong G_2\cong G$ there exists $G_0\leq A$ such that $G_0\cong G$ and  $A=G_0\oplus H=G_0\oplus K$. Every group with the substitution property has the cancellation property. Warfield proved in \cite{Warf} that a group $G$ has the substitution property if and only if the endomorphism ring of $G$, $\End(G)$, has $1$ in the stable range. The reader can find more details about this condition in \cite{Ar}, \cite{Fuchs-2015}, and \cite{Facchini}. In particular, groups (modules) with semilocal endomorphism rings have the substitution property. 

A ring $R$ has \textit{the unit lifting property} if  for all positive integers $n$ the units in $R/nR$ lift to units in $R$. It was proved by Stelzer in \cite[Theorem A]{Stel} that if a finite rank torsion-free group $G$ without free direct summands has the cancellation property then $\End(G)$ has the unit lifting property. As a consequence, he concludes that if the finite rank torsion-free group $G\ncong \BBZ$ is strongly indecomposable (i.e. the ring $\BBQ\End(G)$ is local) then $G$ has the cancellation property if and only if it has the substitution property. Blazhenov used Stelzer's theorem in \cite[Theorem 21]{Bla} to prove that a finite rank torsion-free group $G$ without free direct summands has the cancellation property if and only if (i) for every positive integer $n$ the units of $\End(G)/n\End(G)$ can be lifted to units of $\End(G)$, and  (ii) the endomorphism rings of all quasi-direct summands of $G$ satisfy the conclusion of two theorems proved by Eichler in the 1930s. 

In fact Stelzer proved that if $G$ is a torsion-free group of finite rank, and it has no direct summands isomorphic to $\BBZ$ or $\BBQ$ then there exists a group $A$ such that  (a) $\Hom(A,G)=0=\Hom(G,A)$, (b) $\End(A)\cong \BBZ$, and (c) for every positive integer $n$ there is an epimorphism $A\to G/nG$.
Then the conclusion comes from a result of Fuchs, \cite{Fuchs-canc}, which states that if $G$ has this  property and the cancellation property then $\End(G)$ has the unit lifting property. 

The aim of this paper is to apply the technique introduced by Fuchs in \cite{Fuchs-canc} and presented by Arnold in \cite[Section 8]{Ar} to self-small mixed groups of finite torsion-free rank. Even the study of some decomposition properties is difficult for mixed groups, \cite[p.13, Remark (e)]{Ka}, the restriction of such investigations to \textit{the class $\CS$} of self-small groups of finite torsion-free rank has many advantages. For instance, if $G\in \CS$ then \textit{the Walk-endomorphism ring} $\End_\mathrm{W}(G)=\End(G)/\Hom(G,T(G))$ is torsion-free of finite rank and every idempotent from $\End_\mathrm{W}(G)$ can be lifted to an idempotent endomorphism of $G$, see \cite[Theorem 3.4]{Br-semi} and \cite[Corollary 2.4]{ABW-09}. These properties were used to prove that that if the Walk-endomorphism ring of a group $G\in \CS$ is semi-local then $G$ has the cancellation property, \cite{Br-semi}. Moreover, every group in $\CS$ has a unique, up to quasi-isomorphism, quasi-decomposition as a direct sum of strongly indecomposable self-small groups, \cite{Br-qd}. We refer to \cite{ABW-09} and to \cite{Br-S} for other properties of self-small groups.

The main difference in the case of mixed groups $G$ is that there may be no groups $A$ that satisfy the conditions (a)--(c) described above. We will prove in Proposition \ref{f-lemma} that these conditions can be modified so that the Fuchs' technique still works. In the end of the paper we will prove, by using some similar techniques to \cite{ABVW},  that this version of Fuchs' lemma can be applied to (mixed) quotient-divisible groups. 

The set of all primes is denoted by $\mathbb{P}$. If $G$ is a group then $T(G)$ will be the torsion part of $G$, $T_p(G)$ will denote the $p$-component of $G$ ($p\in \mathbb{P}$), and we write $\overline{G}=G/T(G)$. If $P\subseteq \mathbb{P}$  then $T_P(G)=\oplus_{p\in P}T_p(G)$, and if $n$ is a positive integer we will denote by $T_n(G)$ the subgroup $\oplus_{p\mid n,\ p\in\BBP}T_n(G)$. If $f\in \End(G)$ then $\overline{f}\in \End_\rmW(G)$ represents the coset of $f$ modulo $\Hom(G,T(G))$.

\section{Self-small groups}\label{sect-self-small}

A group $G$ is \textit{self-small} if for every index set $I$, the natural homomorphism $\Hom(G, G)^{(I)}\to \Hom(G, G^{(I)})$ is an isomorphism. We denote by $\CS$ the class of self-small groups of finite torsion-free rank. 

\begin{theorem}\cite[Theorem 2.1]{ABW-09}\label{char-ss}
Let $G$ be a group of finite torsion-free rank. The following are equivalent:
\begin{enumerate}[{\rm 1)}]
 \item $G\in \CS$;
 \item for all $p\in\BBP$ the $p$-components $T_p(G)$ are finite, and $\Hom(G,T(G))$ is a torsion group;
 \item for every $p\in\BBP$ the $p$-component $T_p(G)$ is finite and if $F_G\leq G$ is a full free subgroup of $G$ then $G/F_G$ is $p$-divisible for almost all $p\in\BBP$ such that $T_{p}(G)\neq 0$.
\end{enumerate}
\end{theorem}

Let $G\in\CS$. It follows that the Walk-endomorphism ring of $G$ is the quotient ring $\End_\rmW(G)=\End(G)/T(\End(G))$, and it is torsion-free of finite rank. 
Moreover, for every positive integer $n$ the subgroup $T_n(G)$ is a direct summand of $G$, and the image of every homomorphism $G\to T(G)$ is finite.
\textit{We fix a direct decomposition $G=T_n(G)\oplus G(n)$, and we denote by $\pi_n:G\to G(n)$ and $\upsilon_n:G(n)\to G$ the canonical projection and the canonical injection induced by this decomposition.} 
We note that from the proof of \cite[Proposition 1.1]{Br-semi} it follows that $\End_\rmW(G)$ is $p$-divisible if and only if $G(p)$ is $p$-divisible, and this is equivalent to $\overline{G}$ is $p$-divisible.

\begin{lemma}\label{end-ss} Let $G\in \CS$. If  $k$ is a positive integer and $\theta:G\to G(k)$ is an epimorphism then $\Ker(\theta)=T_k(G)$ and the induced map $\pi_k\theta\upsilon_k:G(k)\to G(k)$ is an isomorphism. 
\end{lemma}

\begin{proof}
Since $\theta$ is an epimorphism it follows that $$\overline{\theta}:G/T(G)\to G(k)/T(G(k)),\ \overline{\theta}(g+T(G))=\theta(g)+T(G(k))$$ is an epimorphism. Since $G/T(G)\cong G(k)/T(G(k))$ is torsion-free of finite rank, we obtain that $\overline{\theta}$ is an isomorphism, so for every element $g\in G$ of infinite order the image $\theta(g)$ is of infinite order. Therefore, for every prime $p\nmid k$ we have $\theta^{-1}(T_p(G(k)))=T_p(G)$. Since all $p$-components of $G$ are finite, we obtain that the induced morphism $\theta:T_p(G)\to T_{p}(G(k))$ is an isomorphism. From all these we obtain $\Ker(\theta)=T_k(G)$. The last statement is now obvious. 
\end{proof}

\section{The lifting property for groups with the cancellation property}

The main aim of this section is to prove a version for the class $\CS$ of Fuchs' Lemma presented in \cite[Lemma 8.10]{Ar}. The main idea used by Fuchs is that the pullback $M$ induced by the canonical projection $G\to G/nG$ and an epimorphism $H\to G/nG$ can be perturbed by using a unit of the ring $\End(G)/n\End(G)$. We obtain a group $M'$ such that $G\oplus M\cong G\oplus M'$ and the cancellation property together with the properties of $H$ lead to the conclusion that $\alpha$ can be lifted to an endomorphism of $G$. This technique was also used for the study of cancellation properties of finitely generated modules over noetherian domains, \cite{Wie}. 

We say that an epimorphism $\alpha:H\to L$ is \textit{rigid} if for every commutative diagram 
$$\xymatrix{ H \ar[r]^{\alpha}\ar[d]^\psi & L\ar[d]^\phi \\
H\ar[r]^\alpha & L
}$$ 
such that $\psi$ and $\phi$ isomorphisms we have $\phi=\pm 1_L$. 
It is easy to see that if $\End(H)\cong \BBZ$ then all epimorphisms $H\to L$ are rigid.

The promised mixed version for Fuchs' lemma is the following:  

\begin{proposition}\label{f-lemma}
Let $G$ be a self-small group of finite torsion-free rank. Suppose that $n$ is a positive integer such that there exists a torsion-free group $H$ with the following properties 
\begin{enumerate}[{\rm (I)}]
 \item there exists a rigid epimorphism $\alpha:H\to G(n)/nG(n)$,
 \item $\Hom(G,H)=0$, and
 \item $\Hom(H,G)$ is a torsion group.
\end{enumerate}
If $G$ has the cancellation property then every unit of $\End_{\mathrm{W}}(G)/n\End_{\mathrm{W}}(G)$ lifts to a unit of $\End_{\mathrm{W}}(G)$.    
\end{proposition}

\begin{proof}
It is enough to assume that $nG(n)\neq G(n)$ (otherwise the ring $\End_{\mathrm{W}}(G)/n\End_{\mathrm{W}}(G)$ is trivial).
If $f$ is an endomorphism of $G$ then $f$ and $(0_{T_n(G)}\oplus \pi_n f \upsilon_n)$ induce the same Walk-endomorphism of $G$. It follows that all units of $\End_\rmW(G)/ n\End_\rmW(G)$ can be lifted to units of $\End_\rmW(G)$ if and only if all units of $\End_\rmW(G(n))/n\End_\rmW(G(n))$ can be lifted to units of $\End_\rmW(G(n))$. Hence \textit{we can suppose w.l.o.g that $T_n(G)=0$}. We will work in the following setting.

\smallskip

\noindent\textbf{Setting:}
\begin{enumerate}[{\rm (i)}] 
 \item $T_n(G)=0$;
\item we fix two endomorphisms $f,g\in\End(G)$ such that for the induced Walk-endomorphisms $\overline{f}$ and $\overline{g}$ we have $\overline{f}\, \overline{g}+n\End_\rmW(G)=\overline{g}\, \overline{f}+n\End_\rmW(G)=\overline{1_G}+n\End_\rmW(G);$ 
\item $\rho:G\to G/nG$ is the canonical projection and $\rho'=\rho f$.
\end{enumerate}

\begin{lemma}\label{lema-proiectii}
The morphism $\rho'$ is surjective, and $\Ker(\rho')=nG$.  
\end{lemma}

\begin{proof}
There exists $h\in \End(G)$ such that $\overline{f}\, \overline{g}=\overline{1_G}+n\overline{h}$, hence the image of 
$fg-1_G+nh$ is a torsion subgroup of $G$. Since $T(G)$ is $n$-divisible, it follows that the image of $fg-1_G$ 
is contained in $nG$. Then $\rho(fg-1_G)=0$, so $\rho'g=\rho fg =\rho$. 
The inclusion $nG\subseteq \Ker(\rho')$ is obvious. Conversely, if $\rho'(x)=0$ then $f(x)\in nG$, so $gf(x)\in nG$. Since $(fg-1_G)(G)\subseteq nG$, it follows that $x\in nG$, and the proof is complete.
\end{proof}

\begin{lemma}\label{M+G}
Suppose that $\alpha:H\to G/nG$ is an epimorphism (not necessarily rigid). Let $M$ be the pullback of the diagram $G\overset{\rho}\longrightarrow G/nG\overset{\alpha}\longleftarrow H$ and let $M'$ be the pullback of the diagram $G\overset{\rho'}\longrightarrow G/nG\overset{\alpha}\longleftarrow H$. 
\begin{enumerate}[{\rm (a)}]
 \item $G\oplus M\cong G\oplus M'$;
 \item If $M\cong M'$ and  $H$ satisfies the conditions {\rm (II)} and {\rm (III)} then there exist  $\theta_H:H\to H$,  $\theta_G:G\to G$, and $\phi:G/nG\to G/nG$ such that:
\begin{enumerate}[{\rm (i)}]
 \item $\theta_H$ and $\phi$ are automorphisms and $\theta_G$ is a unit of $\End_{\rmW}(G)$;
 \item $\alpha\theta_H=\phi\alpha$ and $\rho\theta_G=\phi\rho'$.
\end{enumerate}
\end{enumerate}
\end{lemma}

\begin{proof}
(a) The above mentioned pullbacks induce the solid part of the following commutative diagram:

\[\xymatrix{
0 \ar[r] & nG\ar[r] \ar@{=}[d] & M \ar[r]^{\gamma}\ar[d]^{\beta}\ar@{-->}@/^1.5pc/[ddd]^{\sigma}
\ar@{}[dr]|{(\mathrm{A})}  & H\ar[d]^{\alpha}\ar[r] & 0 \\
0 \ar[r] & nG\ar[r]  & G \ar[r]^{\rho} \ar@/^/[d]^{g} & G/nG \ar[r]\ar@{=}[d] & 0 \\
0 \ar[r] & nG\ar[r]  & G \ar[r]^{\rho'}\ar@/^/[u]^{f} \ar@{}[dr]|{(\mathrm{B})} & G/nG \ar[r] & 0 \\
0 \ar[r] & nG\ar[r] \ar@{=}[u] & M' \ar[r]^{\gamma'}\ar[u]^{\beta'}\ar@{-->}@/^1.5pc/[uuu]^{\sigma'}  & H\ar[u]^{\alpha}\ar[r] & 0 \\
}.\]
Using the pullback square $(\mathrm{A})$ and the equalities $\rho f \beta'=\rho'\beta'=\alpha\gamma'$ it follows that there exists $\sigma':M'\to M$ such that $f\beta'=\beta\sigma'$ and $\gamma'=\gamma\sigma'$. In the same way, using the square $(\mathrm{B})$ we obtain a morphism $\sigma: M\to M'$ such that $\beta'\sigma=g\beta$ and $\gamma=\gamma'\sigma$.

If we construct the pullback of the top and the bottom short exact sequences from the previous diagram we obtain the commutative diagram 

\[\xymatrix{
 &   nG \ar@{=}[r]\ar@{>->}[d]  & nG\ar@{>->}[d]   \\
nG\ar@{>->}[r]  & K \ar@{->>}[r]\ar@{->>}[d]  & M'\ar@{->>}[d]^{\gamma'}\ar@{-->}@/^/[dl]^{\sigma'} \\
 nG\ar@{>->}[r] \ar@{=}[u] & M \ar@{->>}[r]_{\gamma}  \ar@{-->}@/^/[ur]^{\sigma}  & H \\
}\]
whose horizontal and vertical lines are short exact sequences. Since $\gamma=\gamma'\sigma$ and $\gamma'=\gamma\sigma'$, it follows that the top horizontal short exact sequence and the left side vertical short exact sequence split. Then $K\cong nG\oplus M\cong nG\oplus M'$. 


(b) 
Let $\mu:M'\to M$ be an isomorphism. 
Since $\Hom(G,H)=0$, the solid part of the diagram 
\[\xymatrix{
0 \ar[r] & nG\ar[r]\ar@{-->}[d]^{\mu_1}  & M' \ar[r]^{\gamma'}\ar[d]^{\mu}  & H\ar[r]\ar@{-->}[d]^{\theta_H} & 0 \\
0 \ar[r] & nG\ar[r]  & M \ar[r]^{\gamma}  & H\ar[r] & 0 
},\]
can be completed with $\theta_H:H\to H$ and $\mu_1:nG\to nG$
such that the obtained diagram is commutative. 
Let $\mu_1:nG\to nG$ be the induced map such that the above diagram is commutative. Since $\mu_1$ is injective, it follows that $r_0(\Ima(\mu_1))=r_0(nG)$, hence $\Coker(\mu_1)\cong \Ker(\theta_H)$ is a torsion group. But $H$ is torsion-free, so we obtain $\Coker(\mu_1)\cong \Ker(\theta_H)=0$. We obtain that $\theta_H$ and $\mu_1$ are isomorphisms.

Let $\iota_{\Ker(\beta)}:\Ker(\beta)\to M$ and $\iota_{\Ker(\beta')}:\Ker(\beta')\to M'$ be the inclusion maps. 
Since $\Ker(\beta')\cong \Ker(\alpha)$ and $\Ker(\alpha)$ is a subgroup of finite index in $H$, it follows that $\Hom(\Ker(\beta'),G)$ is a torsion group. Then 
 $\beta\mu\iota_{\Ker(\beta')}$ is of finite order, and there exists a positive integer $k$ such that $\beta\mu\iota_{\Ker(\beta')}(\Ker(\beta'))\leq T_k(G)$. Since $\pi_k\beta\mu\iota_{\Ker(\beta')}=0$, there exists $\theta_{G(k)}:G\to G(k)$ such that $\theta_{G(k)}\beta'=\pi_k\beta\mu$. Moreover, there exists $\mu_2:\Ker(\beta')\to \Ker(\beta)$ such that the diagram 
\[\xymatrix{
0 \ar[r] & \Ker(\beta')\ar[r]^{\iota_{\Ker(\beta')}}\ar@{-->}@/_1.5pc/[dd]_(.3){\mu_2}  & M' \ar[r]^{\beta'}\ar[d]^{\mu}  & G\ar[r]\ar@{-->}@/^1.5pc/[dd]^(.3){\theta_{G(k)}} & 0 \\
0 \ar[r] & \Ker(\beta)\ar[r]^{\iota_{\Ker(\beta)}}  & M \ar[r]^{\beta}\ar@{=}[d]  & G\ar[r]\ar[d]_{\pi_k} & 0 \\
0 \ar[r] & \Ker(\pi_k\beta)\ar[r]  & M \ar[r]^{\pi_k\beta}  & G(k)\ar[r] & 0
}\]
is commutative. From Lemma \ref{end-ss} we obtain that $\Ker(\theta_{G(k)})\cong T_k(G)$, and that the restriction of $\theta_{G(k)}$ to $G(k)$ is an isomorphism. Then $\theta_G=\upsilon_k\theta_{G(k)}$ is a unit of $\End_{\rmW}(G)$, where $\upsilon_k:G(k)\to G$ is the inclusion map.

We constructed $\theta_H$ and $\theta_G$ such that (i) is valid. In order to prove that they verify (ii),  we include all these data in the following diagram 

\[\xymatrix{
 \Ker(\beta')\ar@{>->}[d]_{\iota_{\Ker(\beta')}} \ar@{=}[rr] && \Ker(\alpha)\ar@{>->}[d]^{\iota_{\Ker(\alpha)}}\\
 M'\ar@{->>}[rr]^{\gamma'}\ar[d]_{\beta'} \ar@/_2pc/[ddd]_{\mu} && H\ar[d]^{\alpha} \ar@/^2pc/[ddd]^{\theta_H} \\
G\ar[rr]^{\rho'}\ar@{=}[d]\ar@/_/[d]_{f} \ar[rd]^{\theta_{G(k)}} && G/nG\ar@{=}[d]\ar@{=}[d]\ar@{-->}@/^1pc/[d]^{\phi} \\
G\ar[r]^{\pi_k} &G(k) \ar@/^/[l]^{\upsilon_k}\ar[r]^{\widehat{\rho}}& G/nG \\
M\ar[u]^{\beta} \ar@{->>}[rr]^\gamma && H\ar[u]_{\alpha},
 }\]
where $\widehat{\rho}$ is the restriction of $\rho$ to $G(k)$. 
We have 

\begin{align*}\alpha\theta_H\iota_{\Ker(\alpha)}&=\alpha\theta_H\gamma'\iota_{\Ker(\beta')}=\alpha\gamma\mu\iota_{\Ker(\beta')}=\rho\beta\mu\iota_{\Ker(\beta')}\\ &=\widehat{\rho}\pi_k\beta\mu\iota_{\Ker(\beta')}=
\widehat{\rho}\theta_{G(k)}\beta\iota_{\Ker(\beta')}=0,
\end{align*} 
so there exists $\phi:G/nG\to G/nG$ such that $\phi\alpha=\alpha\theta_H$. 
Moreover, $\widehat{\rho}\theta_{G(k)}\beta'=\widehat{\rho}\pi_k\beta\mu=\rho\beta\mu=\alpha\gamma\mu=\alpha\theta_H\gamma'=
\phi\alpha\gamma'=\phi\rho'\beta'.$ But $\beta'$ is surjective, so $\widehat{\rho}\theta_{G(k)}=\phi\rho'$. Then, for $\theta_G=\upsilon_k \theta_{G(k)}$, we have $\rho \theta_G=\phi\rho'$. Since $\phi$ is epic and $G/nG$ is finite, we obtain that $\phi$ is an isomorphism.
\end{proof}

We return to the proof of Proposition \ref{f-lemma}. Let $\theta_G$ and $\phi$ the morphisms constructed in Lemma \ref{M+G}. By (I) $\alpha$ is rigid, and it follows that $\phi=\pm 1_{G/nG}$. Then $\rho\theta_G=\pm \rho'$, and the proof is complete.
\end{proof}

Stelzer proved in \cite[Theorem A]{Stel} that every reduced finite rank torsion-free group $G$ without free direct summands satisfies the hypothesis of Proposition \ref{f-lemma}. Consequently, if $G$ has the cancellation property the endomorphism ring $\End(G)$ has the unit lifting property, \cite[Theorem]{Stel}. Similar results can be extracted  for quotient-divisible groups from the proof of \cite[Theorem 3.4]{ABVW}. We include here a direct proof for this case.
A group of finite torsion-free rank $G$  is \textit{quotient-divisible} if its torsion part is reduced and there exists a full free subgroup $F\leq G$ such that $G/F$ is divisible. It is easy to see that every quotient-divisible group is self-small. We refer to \cite{FoW98} for more details about the structure of mixed quotient-divisible groups. 

\begin{proposition}\label{qd}
Every reduced quotient-divisible group $G$ without free direct summands satisfies the hypothesis of Proposition \ref{f-lemma}. Consequently, if $G$ has the cancellation property then $\End_\rmW(G)$ has the unit lifting property.
\end{proposition}

\begin{proof} 
It was proved in \cite[Proposition 3.3]{ABVW} that there exists an uncountably family $\CW$ of  torsion-free groups of rank $m$ such that for all $W\in \CW$ we have $\End(W)\cong \BBZ$, 
$\Hom(G,W)=0$, and $\Hom(W,G)$ is a torsion group.
The groups from $\CW$ are constructed in \cite[Lemma 4.1]{Go-W}. They are quotient-divisible, torsion-free, and homogeneous of type $0$. 
If $W\in \CW$ and $V\neq 0$ is a pure subgroup of $W$ then $W/V$ is divisible, \cite[Theorem 2.1]{Go-W}. Then 
$\Hom(W_1,W_2)=0$ for all $W_1,W_2\in \CW$ with $W_1\neq W_2$.

Moreover, if $W\in\CW$ then it is of $p$-rank $1$ for all $p\in \BBP$. It follows that for every finite cyclic group $\langle u\rangle$ there exists an epimorphism $W\to \langle u\rangle$. 

Let $n$ be positive integer. Then $G(n)/nG(n)$ is a finite group, and we take a decomposition $G(n)/nG(n)=\oplus_{i=1}^t\langle u_i\rangle$. 
For every $i\in\{1,\dots,t\}$ we consider an epimorphism $\alpha_{i}:W_i\to \langle u_i\rangle$, and for every pair $i,j\in\{1,\dots,t\}$ with $i< j$ an epimorphism $\alpha_{ij}:W_{ij}\to \langle u_i+u_j\rangle$. Since the family $\CW$ is infinite, we can take the groups $W_i$ and $W_{ij}$, $i,j\in\{1,\dots,t\}$, such that there are no non-trivial morphisms between two such groups.

We denote $H=\left(\oplus_{i=1}^{t}W_i\right)\oplus\left(\oplus_{1\leq i<j\leq t} W_{ij}\right)$, and we consider the epimorphism $\alpha:H\to G(n)/nG(n)$  induced by $\alpha_i$ and $\alpha_{ij}$. In order to complete the proof it is enough to prove that $\alpha$ is rigid.

Let $\theta:H\to H$ be and automorphism and $\phi:G(n)/nG(n)\to G(n)/nG(n)$ such that $\alpha\theta=\phi\alpha$ (since $\phi:G(n)/nG(n)$ is finite, $\phi$ is also an automorphism). By the choice of the groups $W_i$ and $W_{ij}$, it follows that $\theta=\left(\oplus_{i=1}^{t}\theta_i\right)\oplus\left(\oplus_{1\leq i<j\leq t} \theta_{ij}\right)$, where $\theta_i=\pm 1_{W_i}$ and $\theta_{ij}=\pm 1_{W_{ij}}$ for all indexes $i$ and $j$. It is easy to see that for every $i$ and $j$ we have $\phi(u_i)=\phi\alpha(w_i)=\alpha\theta(w_i)=\alpha(\pm w_i)=\pm u_i$, where $w_i\in L_i$ is a suitable element. In the same way, for all $1\leq i<j\leq t$ we have $\phi(u_i+u_j)=\pm(u_i+u_j)$.
Suppose that there exist $i\neq j\in\{1,\dots,t\}$ such that $\phi(u_i)=u_i$ and $\phi(u_j)=-u_j$. We can suppose w.l.o.g. that $i<j$.  If $\theta_{ij}=1_{W_{ij}}$ then $\phi(u_i+u_j)=u_i+u_j$, and it follows that $2u_j=0$. Then $\phi(u_j)=u_j$. If $\theta_{ij}=-1_{W_{ij}}$ then $\phi(u_i+u_j)=-(u_i+u_j)$, it follows that 
$\phi(u_i)=-u_i$. It follows that $\phi=\pm 1_{G(n)/nG(n)}$, and the proof is complete. 
\end{proof}

Using \cite[Theorem 2.3]{ABVW} we obtain the following

\begin{corollary}\cite[Theorem 3.4]{ABVW}
Suppose that $G$ is a quotient divisible group such that $\BBQ\End(G)$ is local. Then $G$ has the cancellation property if and only if $G\cong B\oplus\BBZ$, where $B$ is a finite group, or $G$ has the substitution property. 
\end{corollary}

\begin{remark}
It is an open problem to decide if all reduced groups from $\CS$ without free direct summands verify the hypothesis of Proposition \ref{f-lemma}. In \cite{Br-20} it is shown that this property is also valid for some classes of groups that are not necessarily torsion-free nor quotient-divisible. These classes include the class of groups of torsion-free rank at most $3$.  
\end{remark}

\end{document}